\documentclass[10.5pt,reqno]{article}
\usepackage{amsmath,amssymb,amsthm}
\usepackage[english]{babel}
\newtheorem{thm}{Theorem}[section]
\newtheorem{lem}[thm]{Lemma}
\newtheorem{cor}[thm]{Corollary}
\newtheorem{defn}[thm]{Definition}

\newtheorem{rem}[thm]{Remark}
\numberwithin{equation}{section}

\title{The Cauchy problem for the energy-critical inhomogeneous nonlinear Schr\"{o}dinger equation with inverse-square potential}
\author{{\bf RoeSong Jang, JinMyong An, JinMyong Kim$^*$}\\
\footnotesize{Faculty of Mathematics, {\bf Kim Il Sung} University, Pyongyang, Democratic People's Republic of Korea}\\
\footnotesize{$^*$ Corresponding Author}\\
\footnotesize{Email address: jm.kim0211@ryongnamsan.edu.kp}
}
\date{}
\begin{document}
\maketitle
\begin{abstract}
In this paper, we study the Cauchy problem for the
energy-critical inhomogeneous nonlinear Schr\"{o}dinger equation with inverse-square potential
\[iu_{t} +\Delta u-c|x|^{-2}u=\lambda|x|^{-b} |u|^{\sigma } u,\;
u(0)=u_{0} \in H^{1},\;(t,x)\in \mathbb R\times\mathbb R^{d},\]
where $d\ge3$, $\lambda=\pm1$, $0<b<2$, $\sigma=\frac{4-2b}{d-2}$ and  $c>-c(d):=-\left(\frac{d-2}{2}\right)^{2}$.
We first prove the local well-posedness as well as small data global well-posedness and scattering in $H^{1}$ for $c>-\frac{(d+2-2b)^{2}-4}{(d+2-2b)^{2}}c(d)$ and $0<b<\frac{4}{d}$, by using the contraction mapping principle based on the Strichartz estimates.
Based on the local well-posedness result, we then establish the blowup criteria for solutions to the equation in the focusing case $\lambda=-1$. To this end, we derive the sharp Hardy-Sobolev inequality and virial estimates related to this equation.
\end{abstract}

\noindent {\bf Keywords}: Inhomogeneous nonlinear Schr\"{o}dinger equation; Inverse-square potential; Energy-critical; Well-posedness; Blowup; Hardy-Sobolev inequality; Virial estimates\\

\noindent {\bf MR(2020) Subject Classification}: 35Q55, 35A01, 35B44
\section{Introduction}

In this paper, we consider the Cauchy problem for the inhomogeneous nonlinear
Schr\"{o}dinger equation with inverse-square potential, denoted by INLS$_{c}$ equation,
\begin{equation} \label{GrindEQ__1_1_}
\left\{\begin{array}{l} {iu_{t}-P_{c}u=\lambda |x|^{-b}
|u|^{\sigma } u,\;(t,x)\in\mathbb R\times\mathbb R^{d},}
\\ {u\left(0,\; x\right)=u_{0}
\left(x\right),} \end{array}\right.
\end{equation}
where $d\ge3$, $u:\;\mathbb R\times \mathbb R^{d} \to \mathbb C$, $u_{0} :\;\mathbb R^{d} \to \mathbb C$, $b,\;\sigma >0$, $\lambda=\pm1$ and $P_{c}=-\Delta u+c|x|^{-2}$ with $c>-c(d):=-\left(\frac{d-2}{2}\right)^{2}$. $\lambda=-1$ corresponds to the focusing case
and $\lambda=1$ corresponds to the defocusing case. The restriction on $c$ comes from the sharp Hardy inequality:
\begin{equation}\label{GrindEQ__1_2_}
\frac{(d-2)^{2}}{4}\int_{\mathbb R^{d}}{|x|^{-2}\left|u(x)\right|^{2}dx}\le\int_{\mathbb R^{d}}{\left|\nabla u(x)\right|dx},~\forall u\in H^{1}(\mathbb R^{d}),
\end{equation}
which ensures that $P_{c}$ is a positive operator. The INLS$_{c}$ equation appears in a variety of physical settings, for example, in nonlinear optical systems with spatially dependent interactions (see e.g. \cite{BPVT07} and the references therein). In particular, when $c = 0$, it can be thought of as modeling inhomogeneities in the medium in which the wave propagates (see e.g. \cite{KMVBT17}). When $b = 0$, the equation \eqref{GrindEQ__1_1_} also appears in various areas of physics, for instance in quantum field equations, or in the study of certain black hole solutions of the Einstein equations (see e.g. \cite{BPST03,KSWW75}).

The case $b=c=0$ is the classic nonlinear Schr\"{o}dinger (NLS) equation which has been been widely studied over the last three decades (see e.g. \cite{C03, LP15, WHHG11} and the references therein). The case $b=0$ and $c\neq0$ is known as the NLS equation with inverse-square potential, denoted by NLS$_{c}$ equation, which has also been extensively studied in recent years (see e.g. \cite{D18p,KMVZZ17,KMVZ17,MM18, Y21} and the references therein). Moreover, when $c=0$ and $b\neq0$, we have the inhomogeneous nonlinear Schr\"{o}dinger (INLS) equation, which has also attracted a lot of interest in recent years (see e.g. \cite{AK211,AC21,C21,D18,G17} and the references therein).

On the other hand, the inhomogeneous nonlinear Schr\"{o}dinger with potential in the following form:
\begin{equation} \label{GrindEQ__1_3_}
\left\{\begin{array}{l} {iu_{t}+\Delta u-Vu=\lambda |x|^{-b}
|u|^{\sigma } u,\;(t,x)\in\mathbb R\times\mathbb R^{d},}
\\ {u\left(0,\; x\right)=u_{0}
\left(x\right).} \end{array}\right.
\end{equation}
has also been studied by several authors in recent years.
For example, Dinh \cite{D21} studied the well-posedness, scattering and blowup for \eqref{GrindEQ__1_3_} when $d=3$, $b>0$, $\lambda=\pm1$ and $V$ is a real-valued potential satisfying $V\in K_{0}\cap L^{\frac{3}{2}}$ and $\left\|V_{-}\right\|<4\pi$, where $V_{-}:=\min\left\{V,\;0\right\}$ and $K_{0}$ is defined as  the closure of bounded and compactly supported functions with respect to the Kato norm
$$
\left\|V\right\|_{K}:=\sup_{x\in\mathbb R^{3}}\int_{\mathbb R^3}{\frac{|V(y)|}{|x-y|}dy.}
$$
Luo \cite{L19} also studied the stability and multiplicity of standing waves for \eqref{GrindEQ__1_3_} with $V=|x|^{2}$ (harmonic potential), $\lambda=-1$ and $b<0$.
The case $V(x)=c|x|^{-2}$ with $c>-c(d)$ (inverse-square potential) and $b>0$ was considered by \cite{CG21, S16}.

In this paper, we are interested in \eqref{GrindEQ__1_3_} with $V(x)=c|x|^{-2}$ with $c>-c(d)$ and $b>0$, i.e. we study the INLS$_{c}$ equation \eqref{GrindEQ__1_1_}.

Before recalling the known results for the INLS$_{c}$ equation \eqref{GrindEQ__1_1_} and stating our main results, let us give some information about this equation. The INLS$_c$ equation \eqref{GrindEQ__1_1_} is invariant under the scaling,
$$
u_{\lambda}(t,x):=\lambda^{\frac{2-b}{\sigma}}u\left(\lambda^{2}t,\lambda x\right),~\lambda>0.
$$
An easy computation shows that
$$
\left\|u_{\lambda}(0)\right\|_{\dot{H}^{s}}=\lambda^{s-\frac{d}{2}+\frac{2-b}{\sigma}}\left\|u_0\right\|_{\dot{H}^{s}},
$$
which implies that the critical Sobolev index is given by
\begin{equation}\label{GrindEQ__1_4_}
s_{c}=\frac{d}{2}-\frac{2-b}{\sigma}.
\end{equation}
Note that, if $s_{c}=0$ (alternatively $\sigma=\sigma_{\star}:=\frac{4-2b}{d}$) the problem is known as the mass-critical or $L^{2}$-critical; if $s_{c}=1$ (alternatively $\sigma=\sigma^{\star}:=\frac{4-2b}{d-2}$) it is called energy-critical or $\dot{H}^1$-critical. The problem is known as intercritical (mass-supercritical and energy-subcritical) if $0<s_{c}<1$ (alternatively $\sigma_{\star}<\sigma<\sigma^{\star}$).
On the other hand, solutions to the INLS$_{c}$ equation \eqref{GrindEQ__1_1_} conserve the mass and energy, defined respectively by
\begin{equation} \label{GrindEQ__1_5_}
M\left(u(t)\right):=\int_{\mathbb R^{d}}{\left|u(t, x)\right|^{2}dx},
\end{equation}
\begin{equation} \label{GrindEQ__1_6_}
E_{b,c}\left(u(t)\right):=\int_{\mathbb R^{d}}{\frac{1}{2}\left|\nabla u(t, x)\right|^{2}+\frac{c}{2}|x|^{-2}\left|u(t, x)\right|^{2}+\frac{\lambda
}{\sigma+2} |x|^{-b} \left|u(t, x)\right|^{\sigma +2} dx}.
\end{equation}

Let us recall the known results for the INLS$_{c}$ equation \eqref{GrindEQ__1_1_}. Using the energy method, Suzuki \cite{S16} showed that if
 \footnote[1]{\ Note that the author in \cite{S16} considered \eqref{GrindEQ__1_1_} with $c=-c(d)$. The authors in \cite{CG21} pointed out that the proof for the case $c>-c(d)$ is an immediate consequence of the previous one.}
 $d\ge3$, $0<\sigma<\sigma^{\star}$, $c>-c(d)$ and $0<b<2$, then the INLS$_{c}$ equation \eqref{GrindEQ__1_1_} is locally well-posed in $H^{1}_{c}(\mathbb R^{d})$ (which is equivalent to $H^{1}(\mathbb R^{d}))$. It was also proved that any local solution of \eqref{GrindEQ__1_1_} with $u_{0}\in H^{1}_{c}(\mathbb R^{d})$ extends globally in time if either $\lambda=1$ (defocusing case) or $0<\sigma<\sigma_{\star}$ for $\lambda=-1$ (focusing, mass-subcritical case). Recently, Campos-Guzm\'{a}n \cite{CG21} established the sufficient conditions for global existence and blowup in $H^{1}_{c}(\mathbb R^{d})$ for $d\ge 3$, $\lambda=-1$ and $\sigma_{\star}\le\sigma<\sigma^{\star}$, using a Gagliardo-Nirenberg-type estimate. They also studied the local well-posedness and small data global well-posedness under some assumption on $b$ and $c$ in the energy-subcritical case $\sigma<\sigma^{\star}$ with $d\ge 3$ by using the standard Strichartz estimates combined with the fixed point argument. Furthermore, they showed a scattering criterion and construct a wave operator in $H^{1}_{c}(\mathbb R^{d})$, for the intercritical case.
As mentioned above, the authors in \cite{CG21, S16} studied the local and global well-posedness as well as blowup and scattering in $H^{1}_{c}(\mathbb R^{d})$ with $d\ge3$ for the INLS$_{c}$ equation \eqref{GrindEQ__1_1_} in the energy--subcritical case $\sigma<\sigma^{\star}\left(=\frac{4-2b}{d-2}\right)$.

In this paper, we study the well-posedness and blowup in $H^{1}(\mathbb R^{d})$ with $d\ge 3$ for the INLS$_{c}$ equation \eqref{GrindEQ__1_1_} in the energy-critical case $\sigma=\sigma^{\star}=\frac{4-2b}{d-2}$.

First, we prove the local well-posedness as well as small data global well-posedness and scattering by using the contraction mapping principle based on Strichartz estimates.

\begin{thm}\label{thm 1.1.}
Let $d\ge3$, $0<b<\frac{4}{d}$, $\sigma=\frac{4-2b}{d-2}$ and $c>-\frac{(d+2-2b)^{2}-4}{(d+2-2b)^2}c(d)$. If $u_{0} \in H^{1} (\mathbb R^{d})$, then there exists $T=T(u_{0})>0$ such that \eqref{GrindEQ__1_1_} has a unique solution \begin{equation} \label{GrindEQ__1_7_}
u\in L^{\gamma (r)} \left(\left[-T,\;T\right],\;H^{1,r}(\mathbb R^{d})\right),
\end{equation}
where $\left(\gamma (r),\;r\right)$ is an admissible pair satisfying
\begin{equation} \label{GrindEQ__1_8_}
r=\frac{2d(d+2-2b)}{d^2-2db+4} .
\end{equation}
Moreover, for any admissible pair $\left(\gamma \left(p\right),\;p\right)$, we have
\begin{equation} \label{GrindEQ__1_9_}
u\in L^{\gamma \left(p\right)} \left(\left[-T,\;T\right],\;H^{1,p}(\mathbb R^{d})\right).
\end{equation}
If $\left\| u_{0} \right\| _{\dot{H}^{1}(\mathbb R^{d}) } $ is sufficiently small, then the above solution is global and scatters.
\end{thm}

\begin{rem}\label{rem 1.2.}
\textnormal{Theorem \ref{thm 1.1.} can be seen as the extension of the well-posedness result of NLS$_{c}$ equation (see Proposition 3.3 of \cite{D18p}) to the INLS$_{c}$ equation. }
\end{rem}
\begin{rem}\label{rem 1.3.}
\textnormal{In Theorem \ref{thm 1.1.}, the restriction $b<\frac{4}{d}$ comes from the fractional Hardy inequality (Lemma \ref{lem 3.1.}). And the restriction $c>-\frac{(d+2-2b)^{2}-4}{(d+2-2b)^2}c(d)$ comes from the equivalence of Sobolev spaces $\dot{H}^{1,r}_{c}\sim\dot{H}^{1,r}$.}
\end{rem}

Based on the local well-posedness result above, we study the blowup phenomena for the focusing, energy--critical INLS$_{c}$ equation.

Let $0<b<2$, $c>-c(d)$, and let $C_{HS}(b,c)$ be the sharp constant in the Hardy-Sobolev inequality related to the focusing, energy--critical INLS$_{c}$ equation \eqref{GrindEQ__1_1_}, namely,
\begin{equation}\nonumber
C_{HS}(b,c)=\inf_{f\in \dot{H}^{1}_{c}\setminus\left\{0\right\}}{\frac{\left\|f\right\|_{\dot{H}^{1}_{c}}}
{\left\||x|^{-b}|f|^{\sigma^{\star}+2}\right\|
^{\frac{1}{\sigma^{\star}+2}}_{L^{1}}}}.
\end{equation}
We will see in Lemma \ref{lem 4.1.} that:
\begin{enumerate}
  \item When $-c(d)<c\le0$, the sharp constant $C_{HS}(b,c)$ is attained by the function
  \begin{equation}\label{GrindEQ__1_10_}
  W_{b,c}(x):=\frac{\left[\varepsilon(d-b)(d-2)\beta^{2}\right]^{\frac{d-2}{4-2b}}}
  {\left[\varepsilon+|x|^{(2-b)\beta}\right]^{\frac{d-2}{2-b}}|x|^{\rho}},
  \end{equation}
  with $\beta=1-\frac{2\rho}{d-2}$, for all $\varepsilon>0$ (see \eqref{GrindEQ__2_2_} for the definition of $\rho$).
  \item If $c>0$, $C_{HS}(b,c)\le C_{HS}(b,0)$.
\end{enumerate}

We have the following blowup result for the focusing, energy-critical INLS$_{c}$ equation.

\begin{thm}\label{thm 1.4.}
Let $d\ge 3$, $0<b<\frac{4}{d}$, $\lambda=-1$, $c>-\frac{(d+2-2b)^{2}-4}{(d+2-2b)^2}c(d)$ and $\sigma=\frac{4-2b}{d-2}$. Let $u_{0}\in H^{1}(\mathbb R^{d})$ and $u$ be the corresponding solution to \eqref{GrindEQ__1_1_}. Suppose that either $E_{b,c}(u_{0})<0$, or if $E_{b,c}(u_{0})\ge0$, we assume that $E_{b,c}(u_{0})<E_{b,\bar{c}}(W_{b,\bar{c}})$ and $\left\|u_{0}\right\|_{\dot{H}^{1}_{c}}>\left\|W_{b,\bar{c}}\right\|_{\dot{H}^{1}_{\bar{c}}}$, where $\bar{c}=\min\left\{c,\;0\right\}$. If $xu_{0}\in L^{2}$ or $u_{0}$ is radial, then the solution $u$ blows up in finite time.
\end{thm}
\begin{rem}\label{rem 1.5.}
\textnormal{
\begin{enumerate}
  \item In Theorem \ref{thm 1.4.}, the restrictions on $b$ and $c$ only come from the local well-posedness result (Theorem \ref{thm 1.1.}). If we can prove the local existence of solution for the wider range of $b$ and $c$, the result of Theorem \ref{thm 1.4.} still holds.
  \item Theorem \ref{thm 1.4.} can be seen as the extension of the blowup result of NLS$_{c}$ equation (see Theorem 1.12 of \cite{D18p}) to the INLS$_{c}$ equation.
\end{enumerate}
}\end{rem}

This paper is organized as follows. In Section 2, we recall some useful facts which are used in this paper. In Section 3, we prove Theorem \ref{thm 1.1.}. In Section 4, we derive the sharp Hardy-Sobolev inequality and virial estimates related to INLS$_{c}$ equation to prove Theorem \ref{thm 1.4.}.

\section{Preliminaries}

Let us introduce the notation used throughout the paper. As usual, we use $\mathbb C$, $\mathbb R$ and $\mathbb N$ to stand for the sets of complex, real and natural numbers, respectively. $C>0$ will denote positive universal constant, which can be different at different places. $a\lesssim b$ means $a\le Cb$ for some constant $C>0$. We also write $a\sim b$ if $a\lesssim b \lesssim a$.
We denote by $p'$ the dual number of $p\in \left[1,\;\infty \right]$,
i.e. $1/p+1/p'=1$. As in \cite{WHHG11}, for $s\in \mathbb R$ and $1<p<\infty $, we denote by $H^{s,p} (\mathbb R^{d} )$ and $\dot{H}^{s,p} (\mathbb R^{d} )$ the usual nonhomogeneous and homogeneous Sobolev spaces associated to the Laplacian $-\Delta$. As usual, we abbreviate $H^{s,2} (\mathbb R^{d} )$ and $\dot{H}^{s,2} (\mathbb R^{d} )$ as $H^{s} (\mathbb R^{d})$ and $\dot{H}^{s} (\mathbb R^{d} )$, respectively.
Similarly, we define Sobolev spaces in terms of $P_{c}$ via
$$
\left\|f\right\|_{\dot{H}_{c}^{s,p}(\mathbb R^{d})}=\left\|(P_{c})^{\frac{s}{2}}f\right\|_{L^{r}(\mathbb R^{d})},~
\left\|f\right\|_{H_{c}^{s,p}(\mathbb R^{d})}=\left\|(1+P_{c})^{\frac{s}{2}}f\right\|_{L^{r}(\mathbb R^{d})}.
$$
We also abbreviate $\dot{H}_{c}^{s}(\mathbb R^{d})=\dot{H}_{c}^{s,2}(\mathbb R^{d})$ and $H_{c}^{s}(\mathbb R^{d})=H_{c}^{s,2}(\mathbb R^{d})$. Note that by sharp Hardy inequality \eqref{GrindEQ__1_2_}, we see that
\begin{equation}\label{GrindEQ__2_1_}
\left\|f\right\|_{\dot{H}_{c}^{1}(\mathbb R^{d})}\sim \left\|f\right\|_{\dot{H}^{1}(\mathbb R^{d})}~\textrm{for}~ c>-c(d).
\end{equation}
For $I\subset \mathbb R$ and $\gamma \in \left[1,\;\infty \right]$, we will use the space-time mixed space $L^{\gamma } \left(I,\;X\left(\mathbb R,^{d}
\right)\right)$ whose norm is defined by
\[\left\| f\right\|_{L^{\gamma } \left(I,\;X(\mathbb R^{d})\right)}
=\left(\int _{I}\left\| f\right\| _{X(\mathbb R^{d})}^{\gamma } dt
\right)^{\frac{1}{\gamma } } ,\]
with a usual modification when $\gamma =\infty $, where $X(\mathbb R^{d})$
is a normed space on $\mathbb R^{d} $. Given normed spaces $X$ and $Y$, $X\subset Y$ means that $X$ is continuously embedded in $Y$, i.e. there exists a constant $C\left(>0\right)$ such that $\left\| f\right\| _{Y} \le C\left\| f\right\| _{X} $ for all $f\in X$. If there is no confusion, $\mathbb R^{d} $ will be omitted in various function spaces.

Next, we recall the equivalence between the usual Sobolev space defined by $-\Delta$ and the one defined by $P_{c}$. For convenience, we define the following number:
\begin{equation}\label{GrindEQ__2_2_}
\rho:=\frac{d-2}{2}-\sqrt{\left(\frac{d-2}{2}\right)^{2}+c}.
\end{equation}
\begin{lem}[Equivalence of Sobolev spaces, \cite{KMVZZ18}]\label{lem 2.1.}
Let $d\ge3$, $c>-c(d)$ and $0<s<2$.
\begin{enumerate}
  \item If $1<p<\infty$ satisfies $\frac{s+\rho}{d}<\frac{1}{p}<\min\left\{1,\frac{d-\rho}{d}\right\}$, then $\left\|f\right\|_{\dot{H}^{s,p}}\lesssim\left\|f\right\|_{\dot{H}_{c}^{s,p}}$ for all $f\in C^{\infty}_{0}(\mathbb R^{d}\setminus\left\{0\right\})$.
  \item If $1<p<\infty$ satisfies $\max\left\{\frac{s}{d},\;\frac{\rho}{d}\right\}<\frac{1}{p}<\min\left\{1,\frac{d-\rho}{d}\right\}$, then $\left\|f\right\|_{\dot{H}^{s,p}_{c}}\lesssim\left\|f\right\|_{\dot{H}^{s,p}}$ for all $f\in C^{\infty}_{0}(\mathbb R^{d}\setminus\left\{0\right\})$.
\end{enumerate}
\end{lem}
\begin{rem}\label{rem 2.2.} Let $0<s<2$.
\begin{enumerate}
  \item When $c>0$, $\left\|f\right\|_{\dot{H}^{s,p}_{c}}$ is equivalent to $\left\|f\right\|_{\dot{H}^{s,p}}$, provided that $1<p<\frac{d}{s}$.
  \item When $-c(d)\le c<0$, $\left\|f\right\|_{\dot{H}^{s,p}_{c}}$ is equivalent to $\left\|f\right\|_{\dot{H}^{s,p}}$, provided that $\frac{d}{d-\rho}<p<\frac{d}{s+\rho}$.
\end{enumerate}
\end{rem}
We end this section by recalling the Strichartz estimates for the INLS$_{c}$ equation \eqref{GrindEQ__1_1_}.
\begin{defn}\label{defn 2.3.}
\textnormal{Let $d\ge3$. We say that a pair $(\gamma(p),p)$ is \textbf{admissible}, if
\begin{equation} \label{GrindEQ__2_3_}
2\le p\le \frac{2d}{d-2},~\frac{2}{\gamma (p)} =\frac{d}{2} -\frac{d}{p}.
\end{equation}}
\end{defn}
\begin{lem}[Strichartz estimates, \cite{BM18,BPST03}]\label{lem 2.4.}
Let $d\ge3$ and $c>-c(d)$. Then for any $s\in\mathbb R$ and any admissible pairs $(\gamma(p),p)$, $(\gamma(r),r)$, we have
\begin{equation}\label{GrindEQ__2_4_}
\left\|e^{-itP_{c}}f\right\| _{L^{\gamma (p)} (\mathbb R,\;\dot{H}_{c}^{s,p} )} \lesssim\left\|f \right\| _{\dot{H}_{c}^{s} },
\end{equation}
\begin{equation}\label{GrindEQ__2_5_}
\left\|\int _{0}^{t}e^{-i(t-\tau)P_{c}}f(\tau)d\tau\right\| _{L^{\gamma (p)} (\mathbb R,\;\dot{H}_{c}^{s,p} )} \lesssim\left\| f\right\| _{L^{\gamma (r)'} (\mathbb R,\;\dot{H}_{c}^{s,r'} )}.
\end{equation}
\end{lem}

\section{Local and global well-posedness}

In this section, we prove Theorem \ref{thm 1.1.}.
To establish the nonlinear estimates, we recall the following fractional Hardy inequality which is a direct consequence of Theorem 3.1 of \cite{HYZ12}.
\begin{lem}[Fractional Hardy Inequality]\label{lem 3.1.}
Let $1<p<\infty$ and $0<s<\frac{d}{p} $. Then we have
\[\left\| |x|^{-s} f\right\| _{L^{p}(\mathbb R^{d})} \lesssim\left\| f\right\| _{\dot{H}^{s,p} (\mathbb R^{d})}.\]
\end{lem}

Using Lemma \ref{lem 3.1.}, we have the following nonlinear estimates.
\begin{lem}\label{lem 3.2.}
Let $\bar{r}=\frac{2d}{d-2}$, $r=\frac{2d(d+2-2b)}{d^{2}-2db+4}$, $0<b<\frac{4}{d}$ and $\sigma=\frac{4-2b}{d-2}$. Then we have
\begin{equation}\label{GrindEQ__3_1_}
\left\||x|^{-b}|u|^{\sigma}u\right\|_{\dot{H}^{1, \bar{r}'}}\lesssim\left\|u\right\|_{\dot{H}^{1, r}}^{\sigma+1},
\end{equation}
\begin{equation}\label{GrindEQ__3_2_}
\left\||x|^{-b}|u|^{\sigma}v\right\|_{\dot{H}^{1, \bar{r}'}} \lesssim\left\|u\right\|_{\dot{H}^{1, r}}^{\sigma}\left\|v\right\|_{L^{r}}.
\end{equation}
\end{lem}
\begin{proof}
Noticing that
\begin{equation}\label{GrindEQ__3_3_}
\left|\nabla\left(|x|^{-b}|u|^{\sigma}u\right)\right|\lesssim|x|^{-b-1}|u|^{\sigma+1}+|x|^{-b}|u|^{\sigma}|\nabla u|,
\end{equation}
we have
\begin{eqnarray}\begin{split}\label{GrindEQ__3_4_}
\left\| |x|^{-b} |u|^{\sigma } u\right\| _{\dot{H}_{\bar{r}'}^{1} } =\left\| \nabla\left(|x|^{-b} |u|^{\sigma } u\right)\right\| _{\bar{r}'}
\lesssim \left\| |x|^{-b-1}|u|^{\sigma+1}\right\| _{\bar{r}'}+\left\||x|^{-b}|u|^{\sigma}\nabla u\right\| _{\bar{r}'}.
\end{split}\end{eqnarray}
First we estimate $\left\| |x|^{-b-1}|u|^{\sigma+1}\right\| _{\bar{r}'}$. We can see that
\begin{equation} \nonumber
\frac{1}{\bar{r}'}=\left(\sigma +1\right)\left(\frac{1}{r} -\frac{1}{d} \left(1-\frac{b+1}{\sigma +1} \right)\right) .
\end{equation}
Putting
\begin{equation} \nonumber
\frac{1}{\rho } :=\frac{1}{r} -\frac{1}{d} \left(1-\frac{b+1}{\sigma +1} \right),
\end{equation}
we have $\dot{H}^{1-\frac{b+1}{\sigma +1},r} \subset L^{\rho }$. Here, we use the fact $1-\frac{b+1}{\sigma +1}>0 \Leftrightarrow b<\frac{4}{d}$. Using Lemma 3.1, we have
\begin{equation}\label{GrindEQ__3_5_}
\left\||x|^{-b-1} |u|^{\sigma +1} \right\| _{\bar{r}'} =\left\| |x|^{-\frac{b+1}{\sigma +1} } u\right\| _{\rho }^{\sigma +1} \lesssim\left\| |x|^{-\frac{b+1}{\sigma +1} } u\right\| _{\dot{H}^{1-\frac{b+1}{\sigma +1},r} }^{\sigma +1}\lesssim\left\| u\right\| _{\dot{H}^{1, r}}^{\sigma +1}.
\end{equation}
Next we estimate $\left\| |x|^{-b} |u|^{\sigma } \nabla u\right\| _{\bar{r}'} $. We get
\begin{equation} \nonumber
\sigma \left(\frac{1}{r} -\frac{1}{d} \left(1-\frac{b}{\sigma } \right)\right)+\frac{1}{r} =\frac{1}{r'} .
\end{equation}
Putting
\begin{equation} \nonumber
\frac{1}{\gamma } :=\frac{1}{r} -\frac{1}{d} \left(1-\frac{b}{\sigma } \right),
\end{equation}
and noticing $1-\frac{b}{\sigma } >0$, we have $\dot{H}^{1-\frac{b}{\sigma}, r } \subset L^{\gamma } $. Hence it follows from H\"{o}lder inequality and Lemma 3.1 that
\begin{equation} \label{GrindEQ__3_6_}
\left\| |x|^{-b} |u|^{\sigma } \nabla u\right\| _{\bar{r}'} \le\left\| |x|^{-\frac{b}{\sigma } } u\right\| _{\gamma }^{\sigma } \left\| \nabla u\right\| _{r} \lesssim\left\| |x|^{-\frac{b}{\sigma } } u\right\| _{\dot{H}^{1-\frac{b}{\sigma},r } }^{\sigma } \left\| u\right\| _{\dot{H}_{r}^{1} }\lesssim\left\| u\right\| _{\dot{H}^{1, r}}^{\sigma +1}.
\end{equation}
In view of \eqref{GrindEQ__3_4_}--\eqref{GrindEQ__3_6_}, we immediately have \eqref{GrindEQ__3_1_}.
Similarly we also have
\begin{equation}\nonumber
\left\| |x|^{-b} |u|^{\sigma } v\right\| _{L^{r'} } \lesssim\left\| |x|^{-\frac{b}{\sigma } } u\right\| _{\gamma }^{\sigma } \left\| v\right\| _{r} \le \left\| |x|^{-\frac{b}{\sigma } } u\right\| _{\dot{H}^{1-\frac{b}{\sigma},r} }^{\sigma } \left\| v\right\| _{r}
\lesssim\left\| u\right\| _{\dot{H}^{1, r}}^{\sigma }\left\| v\right\| _{r},
\end{equation}
this concludes the proof.
\end{proof}

\begin{proof}[\textbf{Proof of Theorem 1.1}]

We can easily see that $(\gamma(r),r)$ is admissible, where $r$ is given in (1.8). Furthermore, using Remark \ref{rem 2.2.}, we can easily verify that
$\dot{H}^{1,r}_{c}$ is equivalent to $\dot{H}^{1,r}$ provided that $c>-\frac{(d+2-2b)^{2}-4}{(d+2-2b)^2}c(d)$. Putting $\bar{r}=\frac{2d}{d-2}$, we can also see that $\dot{H}^{1,\bar{r}'}_{c}\sim\dot{H}^{1,\bar{r}'}$.
Noticing
\begin{equation}\label{GrindEQ__3_7_}
\frac{1}{\gamma \left(\bar{r}\right)'}=\frac{\sigma+1}{\gamma (r)},
\end{equation}
and using Lemma \ref{lem 3.2.}, H\"{o}lder inequality, we immediately have
\begin{equation}\ \label{GrindEQ__3_8_}
\left\| |x|^{-b} |u|^{\sigma } u\right\|
_{L^{\gamma \left(\bar{r}\right)^{'} } (I,\;\dot{H}^{1,\bar{r}'})}\lesssim\left\| u\right\| _{L^{\gamma (r)} (I,\;\dot{H}^{1,r})}^{\sigma +1},
\end{equation}
where $I\subset\mathbb R$ is an interval. Using \eqref{GrindEQ__3_7_}, Lemma \ref{lem 3.2.} and H\"{o}lder inequality, we also have
\begin{equation}\label{GrindEQ__3_9_}
\left\| |x|^{-b} |u|^{\sigma } u\right\|
_{L^{\gamma \left(\bar{r}\right)^{'} } (I,\;L^{\bar{r}'})}\lesssim\left\| u\right\| _{L^{\gamma (r)} (I,\;\dot{H}^{1,r})}^{\sigma}\left\| u\right\| _{L^{\gamma (r)} (I,\;L^{r})}.
\end{equation}
In view of \eqref{GrindEQ__3_8_} and \eqref{GrindEQ__3_9_}, we have
\begin{equation}\label{GrindEQ__3_10_}
\left\| |x|^{-b} |u|^{\sigma } u\right\|
_{L^{\gamma \left(\bar{r}\right)^{'} } (I,\;H^{1,\bar{r}'})}\lesssim\left\| u\right\| _{L^{\gamma (r)} (I,\;\dot{H}^{1,r})}^{\sigma }\left\| u\right\| _{L^{\gamma (r)} \left(I,\;H^{1,r} \right)}.
\end{equation}
On the other hand, noticing that
$$
\left||x|^{-b} |u|^{\sigma }
u-|x|^{-b} \left|v\right|^{\sigma } v\right|\lesssim |x|^{-b}(|u|^{\sigma}+|v|^{\sigma})|u-v|,
$$
using  Lemma \ref{lem 3.2.}, we have
\begin{eqnarray}\begin{split} \label{GrindEQ__3_11_}
\left\| |x|^{-b} |u|^{\sigma }
u-|x|^{-b} \left|v\right|^{\sigma } v\right\| _{L^{\bar{r}'} } &\lesssim \left\| |x|^{-b} \left(|u|^{\sigma }
+\left|v\right|^{\sigma } \right)\left(u-v\right)\right\| _{L^{\bar{r}'} }
\\
&\lesssim\left(\left\| u\right\| _{\dot{H}^{1,r}}^{\sigma } +\left\| v\right\| _{\dot{H}^{1,r}}^{\sigma } \right)\left\|u-v\right\| _{L^{r} } .
\end{split}\end{eqnarray}
Using \eqref{GrindEQ__3_7_}, \eqref{GrindEQ__3_11_} and H\"{o}lder inequality, we immediately have
\begin{eqnarray}\begin{split} \label{GrindEQ__3_12_}
&\left\| |x|^{-b} |u|^{\sigma } u-|x|^{-b} \left|v\right|^{\sigma }v\right\|_{L^{\gamma \left(\bar{r}\right)^{'} } (I,\;L^{\bar{r}'})}\\
&~~~~~~~~~~~~~~~~~~~\lesssim \left(\left\| u\right\| _{L^{\gamma (r)} (I,\;\dot{H}^{1,r})}^{\sigma } +\left\| v\right\| _{L^{\gamma (r)} (I,\;\dot{H}^{1,r})}^{\sigma } \right)
\left\|u-v\right\| _{L^{\gamma (r)} (I,\;L^{r})} .
\end{split}\end{eqnarray}

First, we prove the local well-posedness. Let $T>0$ and $M>0$ which will be chosen later. We define the following complete metric space
\begin{equation} \nonumber
D=\left\{u\in L^{\gamma (r)} \left(I,\;H^{1,r} \right):\;\left\| u\right\| _{L^{\gamma (r)} (I,\;H^{1,r})} \le M\right\},
\end{equation}
which is equipped with the metric
\begin{equation} \nonumber
d\left(u,\;v\right)=\left\| u-v\right\| _{L^{\gamma
\left(r\right)} (I,\;L^{r})},
\end{equation}
where  $I=[-T,\;T]$.
We consider the mapping
\begin{equation} \nonumber
T:\;u(t)\to e^{-itP_{c}}u_{0}+\int _{0}^{t}e^{-i(t-\tau)P_{c}}|x|^{-b} \left|u(\tau )\right|^{\sigma } u(\tau )d\tau =: u_{L} +u_{NL}.
\end{equation}
Lemma \ref{lem 2.4.} (Strichartz estimates) yields that
\begin{equation} \label{GrindEQ__3_13_}
\left\| u_{L} \right\| _{L^{\gamma (r)} \left(I,\;H^{1,r} \right)}\sim \left\| u_{L} \right\| _{L^{\gamma (r)} (I,\;H^{1,r}_{c})} \lesssim\left\| u_{0} \right\|_{H^{1}_{c}}\sim\left\| u_{0} \right\|_{H^{1}},
\end{equation}
\begin{equation} \label{GrindEQ__3_14_}
\left\| u_{NL} \right\| _{L^{\gamma (r)} \left(I,\;H^{1,r} \right)}\sim\left\| u_{NL} \right\| _{L^{\gamma (r)} (I,\;H^{1,r}_{c})} \lesssim\left\| |x|^{-b} |u|^{\sigma } u\right\|
_{L^{\gamma \left(\bar{r}\right)^{'} } (I,\;H^{1,\bar{r}'}_{c})},
\end{equation}
\begin{equation} \label{GrindEQ__3_15_}
\left\| Tu-Tv \right\| _{L^{\gamma (r)} (I,\;L^{r})} \lesssim\left\| |x|^{-b} |u|^{\sigma } u-|x|^{-b} \left|v\right|^{\sigma }v\right\|_{L^{\gamma \left(\bar{r}\right)^{'} } \left(I,\;L^{\bar{r}',2} \right)}.
\end{equation}
In view of \eqref{GrindEQ__3_13_}, we can see that $\left\|
u_{L}\right\| _{L^{\gamma (r)} \left(\left[-T,\;T\right], \;H^{1,r}\right)} \to 0$, as $T\to 0$. Take $M>0$ such that $CM^{\sigma } \le
\frac{1}{4} $ and $T>0$ such that
\begin{equation} \label{GrindEQ__3_16_}
\left\| u_{L} \right\| _{L^{\gamma (r)}
\left(\left[-T,\;T\right],\;H^{1,r} \right)} \le \frac{M}{2} .
\end{equation}
Using \eqref{GrindEQ__3_10_}, \eqref{GrindEQ__3_14_}, \eqref{GrindEQ__3_16_}, and the fact $\dot{H}^{1,\bar{r}'}_{c}\sim\dot{H}^{1,\bar{r}'}$, we have
\begin{equation} \label{GrindEQ__3_17_}
\left\| Tu\right\| _{L^{\gamma (r)} \left(I,\;H^{1,r}
\right)}\le \frac{M}{2}+C\left\| u\right\| _{L^{\gamma (r)} \left(I,\;H^{1,r} \right)}^{\sigma+1}\le M.
\end{equation}
In view of \eqref{GrindEQ__3_12_} and \eqref{GrindEQ__3_15_}, we have
\begin{equation} \label{GrindEQ__3_18_}
\left\| Tu-Tv\right\| _{L^{\gamma (r)} \left(I,\;L^{r}
\right)}\le 2CM^{\sigma } \left\| u-v\right\| _{L^{\gamma
\left(r\right)} (I,\;L^{r})} \le \frac{1}{2} \left\|
u-v\right\| _{L^{\gamma (r)} (I,\;L^{r})} .
\end{equation}
\eqref{GrindEQ__3_17_} and \eqref{GrindEQ__3_18_} imply that $T: (D,d)\to (D,d)$ is a contraction mapping. From Banach fixed point theorem, there exists a unique solution $u$ of \eqref{GrindEQ__1_1_} in
$(D,d)$. Furthermore for any admissible pair $\left(\gamma \left(p\right),\;p\right)$, it follows from Lemma \ref{lem 2.4.} (Strichartz estimates) and \eqref{GrindEQ__3_10_} that
\begin{equation} \nonumber
\left\| u\right\| _{L^{\gamma \left(p\right)} \left(I,\;H^{1,p} \right)}
\lesssim \left\| u_{0} \right\| _{H^{1}} +\left\| u\right\| _{L^{\gamma
\left(r\right)} \left(I,\;H^{1,r} \right)}^{\sigma +1},
\end{equation}
which implies $u\in L^{\gamma \left(p\right)} \left(I,\;H^{1, p}
\right)$. This completes the proof of the local well-posedness.\\
Next we prove the global well-posedness with small initial data.
We define the following complete metric space
\begin{equation} \nonumber
E=\left\{u\in L^{\gamma (r)} \left(\mathbb R,,\;H^{1,r} \right):\;\left\| u\right\| _{L^{\gamma (r)} \left(\mathbb R,,\;\dot{H}^{1,r}\right)} \le m,\;\left\| u\right\| _{L^{\gamma (r)} (\mathbb R,\;H^{1,r})} \le M\right\},
\end{equation}
which is equipped with the metric
\begin{equation} \nonumber
d\left(u,\;v\right)=\left\| u-v\right\| _{L^{\gamma
\left(r\right)} (\mathbb R,\;L^{r})}.
\end{equation}
Using Lemma \ref{lem 2.4.} (Strichartz estimates) and \eqref{GrindEQ__3_8_}, it follows from the facts $\dot{H}^{1,r}_{c}\sim\dot{H}^{1,r}$ and $\dot{H}^{1,\bar{r}'}_{c}\sim\dot{H}^{1,\bar{r}'}$ that
\begin{equation} \label{GrindEQ__3_19_}
\left\| Tu\right\| _{L^{\gamma (r)} \left(\mathbb R,,\;\dot{H}^{1,r}
\right)} \le C\left\| u_{0} \right\| _{\dot{H}^{1}} +C\left\| u\right\| _{L^{\gamma (r)}
(\mathbb R,\;\dot{H}^{1,r})}^{\sigma +1} .
\end{equation}
Similarly, using Lemma \ref{lem 2.4.} (Strichartz estimates), \eqref{GrindEQ__3_10_} and \eqref{GrindEQ__3_12_}, we also have
\begin{equation} \label{GrindEQ__3_20_}
\left\| Tu\right\| _{L^{\gamma (r)} \left(\mathbb R,,\;H^{1,r}
\right)} \le C\left\| u_{0} \right\| _{H^{1}} +C\left\| u\right\| _{L^{\gamma (r)}
(\mathbb R,\;\dot{H}^{1,r})}^{\sigma }\left\| u\right\| _{L^{\gamma (r)}
\left(\mathbb R,,\;H^{1,r} \right)},
\end{equation}
\begin{equation} \label{GrindEQ__3_21_}
\left\| Tu-Tv\right\| _{L^{\gamma (r)} \left(\mathbb R,,\;L^{r}
\right)}\le C\left(\left\| u\right\| _{L^{\gamma (r)}
(\mathbb R,\;\dot{H}^{1,r})}^{\sigma } +\left\| v\right\| _{L^{\gamma
\left(r\right)} (\mathbb R,\;\dot{H}^{1,r})}^{\sigma } \right)
\left\|u-v\right\| _{L^{\gamma (r)} \left(\mathbb R,,\;L^{r} \right)} .
\end{equation}
Put $m=2C\left\| u_{0} \right\| _{\dot{H}^{1}} $, $M=2C\left\| u_{0} \right\| _{H^{1}} $ and $\delta
=2\left(4C\right)^{-\frac{\sigma +1}{\sigma } } $. If $\left\| u_{0} \right\|
_{\dot{H}^{1}} \le \delta $, i.e. $Cm^{\sigma } <\frac{1}{4} $, then it follows from \eqref{GrindEQ__3_19_}--\eqref{GrindEQ__3_21_} that
\begin{equation} \nonumber
\left\| Tu\right\| _{L^{\gamma (r)} \left(\mathbb R,,\;\dot{H}^{1,r}
\right)}\le m,
\end{equation}
\begin{equation} \nonumber
\;\left\| Tu\right\| _{L^{\gamma (r)} \left(\mathbb R,,\;H^{1,r}
\right)}\le M,
\end{equation}
\begin{equation} \nonumber
\left\| Tu-Tv\right\| _{L^{\gamma (r)} \left(\mathbb R,,\;L^{r}
\right)} \le \frac{1}{2} \left\| u-v\right\|
_{L^{\gamma (r)} \left(\mathbb R,,\;L^{r} \right)} .
\end{equation}
So $T: (E,d)\to (E,d)$ is a contraction mapping and
there exists a unique solution $u$ in $E$.
The scattering result with small initial data can be proved using the standard argument and we omit the details. This concludes the proof.
\end{proof}

\section{Blowup}
In this section, we prove Theorem 1.4. To arrive at this goal, we derive the sharp Hardy-Sobolev inequality as well as the standard virial identity and localized virial estimate related to the focusing, energy-critical INLS$_{c}$ equation.

\subsection{Sharp Hardy-Sobolev inequality}

In this subsection, we consider the sharp Hardy-Sobolev inequality related to the focusing, energy-critical INLS$_{c}$ equation:
\begin{equation}\label{GrindEQ__4_1_}
\left(\int_{\mathbb R^{d}}{|x|^{-b}\left|f\right|^{\sigma^{\star}+2}dx}\right)^{\frac{1}{\sigma^{\star}+2}}\le C_{HS}(b,c) \left\|f\right\|_{\dot{H}^{1}_{c}},
\end{equation}
where the sharp constant $C_{HS}(b,c)$ is defined by
\begin{equation}\label{GrindEQ__4_2_}
C_{HS}(b,c)=\inf_{f\in \dot{H}^{1}_{c}\setminus\left\{0\right\}}
{\frac{\left\|f\right\|_{\dot{H}^{1}_{c}}}
{\left(\int{|x|^{-b}\left|f\right|^{\sigma^{\star}+2}dx}\right)
^{\frac{1}{\sigma^{\star}+2}}}}.
\end{equation}
\begin{lem}[Sharp Hardy-Sobolev inequality]\label{lem 4.1.}
Let $d\ge 3$, $0<b<2$ and $c>-c(d)$.
\begin{enumerate}
  \item If $-c(d)<c\le0$, then the equality in \eqref{GrindEQ__4_1_} is attained by function $W_{b,c}(x)$ given in \eqref{GrindEQ__1_10_}.
  \item If $c>0$, then $C_{HS}(b,c)\le C_{HS}(b,0)$.
\end{enumerate}
\end{lem}
\begin{proof}
The proof of Item 1 can be found in \cite{KP04}. Using the fact $c>0$, we immediately have that $\left\|f\right\|_{\dot{H}^{1}}<\left\|f\right\|_{\dot{H}^{1}_{c}}$ for any $f\in \dot{H}^{1}\setminus\left\{0\right\}$. Hence it follows from Item 1 that
\begin{equation}\nonumber
\left(\int{|x|^{-b}\left|f\right|^{\sigma^{\star}+2}dx}\right)^{\frac{1}{\sigma^{\star}+2}}\le C_{HS}(b,0) \left\|f\right\|_{\dot{H}^{1}}<C_{HS}(b,0) \left\|f\right\|_{\dot{H}^{1}_{c}},
\end{equation}
which implies that $C_{HS}(b,c)\le C_{HS}(b,0)$. This completes the proof.
\end{proof}
\begin{rem}\label{rem 4.2.}
\textnormal{When $b=0$ and $c>0$, it is known that $C_{HS}(b,c)=C_{HS}(b,0)$ and the equality in \eqref{GrindEQ__4_1_} is never attained. In fact, $C_{HS}(0,c)\ge C_{HS}(0,0)$ can be proved by considering $f_{n}(x)=W_{0,c}(x-x_{n})$ for any sequence $x_{n}\rightarrow\infty$. See \cite{D18p,KMVZZ17} for details. But when $b>0$, we could not apply this argument and we don't know whether $C_{HS}(b,c)= C_{HS}(b,0)$.}
\end{rem}

Next, we recall some properties related to $W_{b,c}$. Lemma 2.2 of \cite{KP04} also shows that $W_{b,c}$ with $-c(d)<c\le0$ solves the equation
\begin{equation}\nonumber
P_{c}W_{b,c}=|x|^{-b}\left|W_{b,c}\right|^{\sigma^{\star}}W_{b,c},
\end{equation}
and satisfies
\begin{equation}\label{GrindEQ__4_3_}
\left\|W_{b,c}\right\|_{\dot{H}^{1}_{c}}^{2}=\int{|x|^{-b} W_{b,c}^{\sigma^{\star}+2}dx}.
\end{equation}
Hence, we have for $-c(d)<c\le0$,
\begin{equation}\label{GrindEQ__4_4_}
\left\|W_{b,c}\right\|_{\dot{H}^{1}_{c}}^{2}=\int{|x|^{-b} W_{b,c}^{\sigma^{\star}+2}dx}=C_{HS}(b,c)^{-\frac{2(d-b)}{2-b}},
~\left\|W_{b,c}\right\|_{\dot{H}^{1}_{c}}^{\sigma^{\star}}=C_{HS}(b,c)^{-(\sigma^{\star}+2)},
\end{equation}
\begin{equation} \label{GrindEQ__4_5_}
E_{b,c}\left(W_{b,c}\right)=\frac{1}{2} \left\|
 W_{b,c}\right\|_{\dot{H}^{1}_{c}}^{2}-\frac{1}{\sigma^{\star}+2} \int{|x|^{-b} W_{b,c}^{\sigma^{\star}+2}dx}=\frac{2-b}{2(d-b)}C_{HS}(b,c)^{-\frac{2(d-b)}{2-b}}.
\end{equation}
Moreover, for any $c>-c(d)$, we have
\begin{equation}\label{GrindEQ__4_6_}
C_{HS}(b,c)\le C_{HS}(b,\bar{c})=\left\|W_{b,\bar{c}}\right\|_{\dot{H}^{1}_{\bar{c}}}^{-\frac{2-b}{d-b}}
=\left\||x|^{-b}W_{b,\bar{c}}^{\sigma^{\star}+2}\right\|_{L^{1}}^{-\frac{2-b}{2(d-b)}}
=\left[\frac{2(d-b)}{2-b}E_{b,\bar{c}}\left(W_{b,\bar{c}}\right)\right]^{-\frac{2-b}{2(d-b)}}.
\end{equation}

\subsection{Virial estimates}

In this subsection, we derive the standard virial identity and localized virial estimate related to the focusing INLS$_{c}$ equation. Given a real valued function $a$, we define the virial potential by
\begin{equation}\nonumber
V_{a}(t):=\int{a(x)\left|u(t,x)\right|^{2}dx.}
\end{equation}
A simple computation shows that the following result holds.
\begin{lem}[\cite{D18p}]\label{lem 4.3.}
Let $d\ge3$ and $c>-c(d)$. If $u:I\times\mathbb R^{d}\rightarrow\mathbb C$ is a smooth-in-time and Schwartz-in-space solution to $iu_{t}-P_{c}u=N(u)$, with $N(u)$ satisfying $\textnormal{Im}(N(u)\bar{u})=0$, then we have for any $t\in I$,
\begin{equation}\nonumber
\frac{d}{dt}V_{a}(t)=2\int{\nabla a(x)\cdot \textnormal{Im}(\bar{u}(t,x)\nabla u(t,x))dx},
\end{equation}
and
\begin{eqnarray}\begin{split}\nonumber
\frac{d^{2}}{dt^{2}}V_{a}(t)&=-\int{\Delta^{2}a(x)|u(t,x)|^{2}dx+4\sum_{j,k=1}^{d}
{\int{\partial_{jk}^{2}a(x)\textnormal{Re}(\partial_{k}u(t,x)\partial_{j}\bar{u}(t,x))dx}}}\\
&+4c\int{\nabla a(x)\cdot\frac{x}{|x|^{4}}|u(t,x)|^{2}dx}+2\int{\nabla a(x)\cdot \left\{N(u),u\right\}_{p}(t,x)dx},
\end{split}\end{eqnarray}
where $\left\{f,g\right\}_p:=\textnormal{Re}(f\nabla \bar{g}-g\nabla\bar{f})$ is the momentum bracket.
\end{lem}
Note that if $N(u)=-|x|^{-b}|u|^{\sigma}u$, then
\begin{equation}\nonumber
\left\{N(u),u\right\}_{p}=\frac{\sigma}{\sigma+2}\nabla(|x|^{-b}|u|^{\sigma+2})+\frac{2}{\sigma+2}
\nabla(|x|^{-b})|u|^{\sigma+2}.
\end{equation}
Hence, we immediately have the following result.
\begin{cor}\label{cor 4.4.}
If $u$ is a smooth-in-time and Schwartz-in-space solution to the focusing INLS$_{c}$ equation, then we have for any $t\in I$,

\begin{eqnarray}\begin{split}\nonumber
\frac{d^{2}}{dt^{2}}V_{a}(t)&=-\int{\Delta^{2}a(x)|u(t,x)|^{2}dx+4\sum_{j,k=1}^{d}
{\int{\partial_{jk}^{2}a(x)\textnormal{Re}(\partial_{k}u(t,x)\partial_{j}\bar{u}(t,x))dx}}}\\
&+4c\int{\nabla a(x)\cdot\frac{x}{|x|^{4}}|u(t,x)|^{2}dx}-\frac{2\sigma}{\sigma+2}
\int{\Delta a(x)|x|^{-b}|u(t,x)|^{\sigma+2}dx}\\
&+\frac{4}{\sigma+2}\int{\nabla a(x)\cdot \nabla(|x|^{-b})|u(t,x)|^{\sigma+2}dx}.
\end{split}\end{eqnarray}
\end{cor}
We have the following standard virial identity for the focusing INLS$_{c}$ equation.
\begin{lem}[Standard Virial Identity]\label{lem 4.5.}
Let $d\ge 3$, $0<b<2$ and $c>-c(d)$. Let $u_{0}\in H^{1}$ be such that $|x|u_{0}\in L^2$ and $u:I\times\mathbb R^{d}\to \mathbb C$ be the corresponding solution to the focusing INLS$_{c}$ equation. Then, $|x|u\in C\left(I,\;L^2\right)$. Moreover, for any $t\in I$,
\begin{equation}\label{GrindEQ__4_7_}
\frac{d^2}{dt^2}\left\|xu(t)\right\|_{L^{2}}^{2}=8 \left\|u(t)\right\| _{\dot{H}^{1}_{c}}^{2} -\frac{4(d\sigma+2b)}{\sigma+2} \int{|x|^{-b} \left|u(t,x)\right|^{\sigma +2}dx.}
\end{equation}
\end{lem}
\begin{proof}
The first claim follows from the standard approximation argument and we omit the details (see e.g. Proposition 6.5.1 of \cite{C03} for details). It remains to prove \eqref{GrindEQ__4_7_}. Applying Corollary \ref{cor 4.4.} with $a(x)=|x|^{2}$, we have
\begin{eqnarray}\begin{split}\nonumber
\frac{d^2}{dt^2}\left\|x u(t)\right\|_{L^{2}}^{2}&=\frac{d^2}{dt^2}V_{|x|^{2}}(t)\\
&=8\int{\left|\nabla u(t,x)\right|^{2}+c|x|^{-2}\left|u(t,x)\right|^{2}dx}
-\frac{4(d\sigma+2b)}{\sigma+2} \int{|x|^{-b} \left|u(t,x)\right|^{\sigma +2}dx}\\
&=8\left\|u(t)\right\| _{\dot{H}^{1}_{c}}^{2} -\frac{4(d\sigma+2b)}{\sigma+2} \int{|x|^{-b} \left|u(t,x)\right|^{\sigma +2}dx,}
\end{split}\end{eqnarray}
this completes the proof.
\end{proof}
Next we derive the localized virial estimate which is used to prove the blowup for the focusing INLS$_{c}$ equation with radial data. To do so, we introduce a function $\theta:\left[0,\;\infty\right)\to \left[0,\;\infty\right)$ satisfying
\begin{equation}\label{GrindEQ__4_8_}
\theta (r)=\left\{\begin{array}{l} {r^2,~\textrm{if}\;0\le r\le 1,}
\\ {\textrm{const},~\textrm{if}\;r\ge 2,} \end{array}\right.
\textrm{and}~~\theta''(r)\le 2,~~\textrm{for}~~r\ge 0.
\end{equation}
For $R>1$, we define the radial function
\begin{equation}\label{GrindEQ__4_9_}
\varphi_{R}(x)=\varphi_{R}(r):=R^{2}\theta(r/R),\;r=|x|.
\end{equation}
One can easily see that
\begin{equation}\label{GrindEQ__4_10_}
2-\varphi''_{R}(r)\ge0,\;2-\frac{\varphi'_{R}(r)}{r}\ge0,\;2d-\Delta\varphi_{R}(x)\ge0.
\end{equation}
\begin{lem}[Localized Virial Estimate]\label{lem 4.6.}
Let $d \ge 3$, $0<b<2$, $c>-c(d)$, $R>1$ and $\varphi_{R}$ be as in \eqref{GrindEQ__4_9_}. Let $u:I\times\mathbb R^{d}\to \mathbb C$ be a radial solution to the focusing INLS$_{c}$ equation. Then for any $\varepsilon>0$ and any $t\in I$,
\begin{eqnarray}\begin{split} \label{GrindEQ__4_11_}
\frac{d^2}{dt^2}V_{\varphi_{R}}(t)&\le 8 \left\|u(t)\right\| _{\dot{H}^{1}_{c}}^{2} -\frac{4(d\sigma+2b)
}{\sigma+2} \int{|x|^{-b} \left|u(t,x)\right|^{\sigma +2}dx}\\
&+O\left(R^{-2}+\varepsilon^{-\frac{\sigma}{4-\sigma}}R^{-{\frac{2\left[(d-1)\sigma+2b\right]}
{4-\sigma}}}+\varepsilon\left\|u(t)\right\| _{\dot{H}^{1}_{c}}^{2}\right).
\end{split}\end{eqnarray}
\end{lem}
\begin{proof}
We use the argument similar to that used to prove the localized virial estimates for INLS equation and NLS$_{c}$ equation (see \cite{D18,D18p}). Applying Corollary 4.3 with $a(x)=\varphi_{R}(x)$, we have
\begin{eqnarray}\begin{split}\nonumber
\frac{d^{2}}{dt^{2}}V_{\varphi_{R}}(t)&=-\int{\Delta^{2}\varphi_{R}(x)|u(t,x)|^{2}dx}+4\sum_{j,k=1}^{d}
{\int{\partial_{jk}^{2}\varphi_{R}(x)\textnormal{Re}(\partial_{k}u(t,x)\partial_{j}\bar{u}(t,x))dx}}\\
&+4c\int{\nabla \varphi_{R}(x)\cdot\frac{x}{|x|^{4}}|u(t,x)|^{2}dx}-\frac{2\sigma}{\sigma+2}
\int{\Delta \varphi_{R}(x)|x|^{-b}|u(t,x)|^{\sigma+2}dx}\\
&+\frac{4}{\sigma+2}\int{\nabla \varphi_{R}(x)\cdot \nabla(|x|^{-b})|u(t,x)|^{\sigma+2}dx}.
\end{split}\end{eqnarray}
Since $\varphi_{R}(x)=|x|^{2}$ for $|x|<R$, it follows from Lemma \ref{lem 4.5.} that
\begin{eqnarray}\begin{split}\nonumber
\frac{d^{2}}{dt^{2}}V_{\varphi_{R}}(t)&=8\left\|u(t)\right\| _{\dot{H}^{1}_{c}}^{2} -\frac{4(d\sigma+2b)}{\sigma+2} \int{|x|^{-b} \left|u(t,x)\right|^{\sigma +2}dx}-8\left\|u(t)\right\| _{\dot{H}^{1}_{c}(|x|>R)}^{2}\\
&+\frac{4(d\sigma+2b)}{\sigma+2} \int_{|x|>R}{|x|^{-b} \left|u(t,x)\right|^{\sigma +2}dx}-\int_{|x|>R}{\Delta^{2}\varphi_{R}(x)|u(t,x)|^{2}dx}\\
&+4\sum_{j,k=1}^{d}
{\int_{|x|>R}{\partial_{jk}^{2}\varphi_{R}(x)\textnormal{Re}(\partial_{k}u(t,x)\partial_{j}\bar{u}(t,x))dx}}\\
&+4c\int_{|x|>R}{\nabla \varphi_{R}(x)\cdot\frac{x}{|x|^{4}}|u(t,x)|^{2}dx}-\frac{2\sigma}{\sigma+2}
\int_{|x|>R}{\Delta \varphi_{R}(x)|x|^{-b}|u(t,x)|^{\sigma+2}dx}\\
&+\frac{4}{\sigma+2}\int_{|x|>R}{\nabla \varphi_{R}(x)\cdot \nabla(|x|^{-b})|u(t,x)|^{\sigma+2}dx}.
\end{split}\end{eqnarray}
Since $\Delta\varphi_{R}\lesssim1$, $\Delta^{2}\varphi_{R}\lesssim R^{-2}$ and $\nabla \varphi_{R}(x)\cdot \nabla(|x|^{-b})\lesssim |x|^{-b}$, we have
\begin{eqnarray}\begin{split}\nonumber
\frac{d^{2}}{dt^{2}}V_{\varphi_{R}}(t)&=8\left\|u(t)\right\| _{\dot{H}^{1}_{c}}^{2} -\frac{4(d\sigma+2b)}{\sigma+2} \int{|x|^{-b} \left|u(t,x)\right|^{\sigma +2}dx}\\
&-8\left\|u(t)\right\| _{\dot{H}^{1}_{c}(|x|>R)}^{2}
+4\sum_{j,k=1}^{d}
{\int_{|x|>R}{\partial_{jk}^{2}\varphi_{R}(x)\textnormal{Re}(\partial_{k}u(t)
\partial_{j}\bar{u}(t))dx}}\\
&+4c\int_{|x|>R}{\nabla \varphi_{R}(x)\cdot\frac{x}{|x|^{4}}|u(t,x)|^{2}dx}\\
&+O\left(\int_{|x|>R}
{R^{-2}|u(t)|^{2}+|x|^{-b}|u(t)|^{\sigma+2}dx}\right).
\end{split}\end{eqnarray}
Using \eqref{GrindEQ__4_10_} and the fact $$\partial^{2}_{jk}=\left(\frac{\delta_{jk}}{r}-\frac{x_{j}x_{k}}{r^{3}}\right)\partial_{r}
+\frac{x_{j}x_{k}}{r^{2}}\partial^{2}_{r},$$ we can see that
$$
4\sum_{j,k=1}^{d}{\partial_{jk}^{2}\varphi_{R}(\partial_{k}u\partial_{j}\bar{u})dx}\le2|\nabla u|^{2},~\nabla\varphi_{R}\cdot x \le 2|x|^{2}.
$$
Hence, we have
\begin{eqnarray}\begin{split}\nonumber
&-8\left\|u(t)\right\| _{\dot{H}^{1}_{c}(|x|>R)}^{2}+4\sum_{j,k=1}^{d}
{\int_{|x|>R}{\partial_{jk}^{2}\varphi_{R}\textnormal{Re}(\partial_{k}u
\partial_{j}\bar{u})dx}}+4c\int_{|x|>R}{\nabla \varphi_{R}(x)\cdot\frac{x}{|x|^{4}}|u|^{2}dx}\\
&~~~~~~~~~~\le 4c\int_{|x|>R}{\left(\nabla \varphi_{R}(x)\cdot x-2|x|^{2}\right)\frac{|u|^{2}}{|x|^{4}}}
=-4c\int_{|x|>R}{\left(2-\frac{\varphi'_{R}(r)}{r}\right)\frac{|u(t)|}{|x|^2}dx}\\
&~~~~~~~~~~\le \max\left\{-4cS,0\right\} \int_{|x|>R}{R^{-2}|u(t)|^{2}dx}\le \max\left\{-4cSM(u(t)),0\right\}R^{-2} ,
\end{split}\end{eqnarray}
where $S=\max_{r\ge 1}{2-\frac{\theta'(r)}{r}}$.
The conservation of mass implies
\begin{eqnarray}\begin{split} \nonumber
\frac{d^2}{dt^2}V_{\varphi_{R}}(t)&\le 8 \left\|u(t)\right\| _{\dot{H}^{1}_{c}}^{2} -\frac{4(d\sigma+2b)
}{\sigma+2} \int{|x|^{-b} \left|u(t,x)\right|^{\sigma +2}dx}\\
&+O\left(R^{-2}+\int_{|x|>R}{|x|^{-b}|u(t)|^{\sigma+2}dx}\right).
\end{split}\end{eqnarray}
Using the same argument as in the proof of Lemma 3.4 of \cite{D18}, it follows from the fact $\dot{H}^{1}\sim\dot{H}^{1}_{c}$, we have
\begin{equation}\nonumber
\int_{|x|>R}{|x|^{-b}|u(t)|^{\sigma+2}dx}\lesssim R^{-\frac{(d-1)\sigma+2b}{2}}\left\|u(t)\right\| _{\dot{H}^{1}_{c}}^{\frac{\sigma}{2}},
\end{equation}
whose proof will be omitted. Next we use the Young inequality
\footnote[2]{\ Let $a$, $b$ be non-negative real numbers and $p$, $q$ be positive real numbers satisfying $\frac{1}{p}+\frac{1}{q}=1$. Then for any $\varepsilon$, we have $ab\lesssim\varepsilon a^{p}+\varepsilon^{-\frac{q}{p}}b^{q}$.}
to get for any $\varepsilon>0$,
$$
R^{-\frac{(d-1)\sigma+2b}{2}}\left\|u(t)\right\| _{\dot{H}^{1}_{c}}^{\frac{\sigma}{2}}\lesssim\varepsilon^{-\frac{\sigma}{4-\sigma}}R^{-{\frac{2\left[(d-1)\sigma+2b\right]}
{4-\sigma}}}+\varepsilon\left\|u(t)\right\| _{\dot{H}^{1}_{c}}^{2},
$$
this completes the proof.
\end{proof}

\subsection{Proof of Theorem \ref{thm 1.4.}}
In this subsection, we prove Theorem \ref{thm 1.4.}.
We divide the study in two cases: $E_{b,c}(u_{0})<0$ and $E_{b,c}(u_{0})\ge0$.
\begin{itemize}
  \item \textbf{The case $E_{b,c}(u_{0})<0$.}
\end{itemize}

First, we consider the case $xu_{0}\in L^{2}$. Applying the standard virial identity \eqref{GrindEQ__4_7_} and the conservation of energy, we have
\begin{eqnarray}\begin{split}\nonumber
\frac{d^2}{dt^2}\left\|x u(t)\right\|_{L^{2}}^{2}&=8 \left\|u(t)\right\| _{\dot{H}^{1}_{c}}^{2} -\frac{4(d\sigma^{\star}+2b)}{\sigma^{\star}+2} \int{|x|^{-b} \left|u(t,x)\right|^{\sigma^{\star} +2}dx}\\
&=4(d\sigma^{\star}+2b)E_{b,c}\left(u(t)\right)-2(d\sigma^{\star}-4+2b)\left\|u\right\|_{\dot{H}^{1}_{c}}^{2}<0,
\end{split}\end{eqnarray}
where we used the fact $d\sigma^{\star}-4+2b>0$.
By the classical argument of Glassey \cite{G77}, it follows that the solution $u$ blows up in finite time.

Next, we consider the case $u_{0}$ is radial.
Using the localized virial estimate \eqref{GrindEQ__4_11_} and the conservation of energy, we have

\begin{eqnarray}\begin{split} \nonumber
\frac{d^2}{dt^2}V_{\varphi_{R}}(t)&\le 8 \left\|u(t)\right\| _{\dot{H}^{1}_{c}}^{2} -\frac{4(d\sigma^{\star}+2b)
}{\sigma^{\star}+2} \int{|x|^{-b} \left|u(t,x)\right|^{\sigma^{\star} +2}dx}\\
&+O\left(R^{-2}+\varepsilon^{-\frac{\sigma^{\star}}{4-\sigma^{\star}}}R^{-{\frac{2\left[(d-1)\sigma^{\star}+2b\right]}
{4-\sigma^{\star}}}}+\varepsilon\left\|u(t)\right\| _{\dot{H}^{1}_{c}}^{2}\right)\\
&=4(d\sigma^{\star}+2b)E_{b,c}\left(u(t)\right)-2(d\sigma^{\star}-4+2b)\left\|u\right\|_{\dot{H}^{1}_{c}}^{2}\\
&+O\left(R^{-2}+\varepsilon^{-\frac{\sigma^{\star}}{4-\sigma^{\star}}}R^{-{\frac{2\left[(d-1)\sigma^{\star}+2b\right]}
{4-\sigma^{\star}}}}+\varepsilon\left\|u(t)\right\| _{\dot{H}^{1}_{c}}^{2}\right),
\end{split}\end{eqnarray}
for any $t$ in the existence time and for any $\varepsilon>0$. Since $d\sigma^{\star}-4+2b>0$, we take $\varepsilon>0$ small enough and $R>1$ large enough depending on $\varepsilon$ to have that
\begin{equation}\nonumber
\frac{d^2}{dt^2}V_{\varphi_{R}}(t)\le2(d\sigma^{\star}+2b)E_{b,c}\left(u_{0}\right)<0,
\end{equation}
for any $t$ in the existence time. This shows that the solution $u$ must blow up in finite time.
\begin{itemize}
  \item \textbf{The case $E_{b,c}(u_{0})\ge0$.}
\end{itemize}

By the definition of the energy \eqref{GrindEQ__1_6_} and Lemma 4.1, we have
\begin{eqnarray}\begin{split}\nonumber
E_{b,c}\left(u(t)\right)&=\frac{1}{2} \left\| u(t)\right\| _{\dot{H}^{1}_{c}}^{2} -\frac{1}{\sigma^{\star}+2} \left\| |x|^{-b} |u|^{\sigma^{\star}+2} \right\|_{L^{1}}\\
&\ge \frac{1}{2} \left\| u(t)\right\| _{\dot{H}^{1}_{c}}^{2} -\frac{C_{HS}(b,\bar{c})^{\sigma^{\star}+2}}{\sigma^{\star}+2}\left\| u(t)\right\| _{\dot{H}^{1}_{c}}^{\sigma^{\star}+2}=
:g\left(\left\| u(t)\right\| _{\dot{H}^{1}_{c}}\right),
\end{split}\end{eqnarray}
where
\begin{equation}\label{GrindEQ__4_12_}
g(y)=\frac{1}{2}y^2-\frac{C_{HS}(b,\bar{c})^{\sigma^{\star}+2}}{\sigma^{\star}+2}y^{\sigma^{\star}+2}.
\end{equation}
It also follows from \eqref{GrindEQ__4_6_} that
\[g(\left\|W_{b,\bar{c}}\right\|_{\dot{H}^{1}_{\bar{c}}})=E_{b,\bar{c}}(W_{b,\bar{c}}).\]
By the conservation of energy and the assumption $E_{b,c}(u_{0})<E_{b,\bar{c}}(W_{b,\bar{c}})$, we can see that
\[g(\left\| u(t)\right\| _{\dot{H}^{1}_{c}})\le E_{b,c}\left(u(t)\right)=E_{b,c}\left(u_{0}\right)<E_{b,\bar{c}}(W_{b,\bar{c}}).\]
By the assumption $\left\|u_{0}\right\|_{\dot{H}^{1}_{c}}>\left\|W_{b,\bar{c}}\right\|_{\dot{H}^{1}_{\bar{c}}}$ and the continuity argument, we have
\begin{equation}\label{GrindEQ__4_13_}
\left\|u(t)\right\|_{\dot{H}^{1}_{c}}>\left\|W_{b,\bar{c}}\right\|_{\dot{H}^{1}_{\bar{c}}},
\end{equation}
for any $t$ as long as the solution exists. \eqref{GrindEQ__4_13_} is improved as follows. Pick $\delta>0$ small enough such that
\begin{equation}\label{GrindEQ__4_14_}
E_{b,c}(u_{0})\le(1-\delta)E_{b,\bar{c}}(W_{b,\bar{c}}),
\end{equation}
which implies that
\begin{equation}\label{GrindEQ__4_15_}
g(\left\| u(t)\right\| _{\dot{H}^{1}_{c}})\le(1-\delta)E_{b,\bar{c}}(W_{b,\bar{c}}).
\end{equation}
Using \eqref{GrindEQ__4_6_}, \eqref{GrindEQ__4_12_} and \eqref{GrindEQ__4_15_}, we have
\begin{equation}\nonumber
\frac{d-b}{2-b}\left(\frac{\left\| u(t)\right\| _{\dot{H}^{1}_{c}}}{\left\| W_{b,\bar{c}}\right\| _{\dot{H}^{1}_{\bar{c}}}}\right)^{2}-\frac{d-2}{2-b}\left(\frac{\left\| u(t)\right\| _{\dot{H}^{1}_{c}}}{\left\| W_{b,\bar{c}}\right\| _{\dot{H}^{1}_{\bar{c}}}}\right)^{\sigma^{\star}+2}\le 1-\delta.
\end{equation}
The continuity argument shows that there exits $\delta'>0$ depending on $\delta$ such that
\begin{equation}\label{GrindEQ__4_16_}
\frac{\left\| u(t)\right\| _{\dot{H}^{1}_{c}}}{\left\| W_{b,\bar{c}}\right\| _{\dot{H}^{1}_{c}}}\ge1+\delta'.
\end{equation}
Then we can take $\varepsilon>0$ small enough such that
\begin{equation}\label{GrindEQ__4_17_}
8\left\| u(t)\right\| _{\dot{H}^{1}_{c}}^{2} -\frac{4(d\sigma^{\star}+2b)}{\sigma^{\star}+2} \left\| |x|^{-b} |u|^{\sigma^{\star}+2} \right\|_{L^{1}}+\varepsilon\left\| u(t)\right\| _{\dot{H}^{1}_{c}}^{2}\le -c<0,
\end{equation}
for any $t$ in the existence time. In fact, using the conservation of energy, \eqref{GrindEQ__4_6_}, \eqref{GrindEQ__4_14_} and \eqref{GrindEQ__4_16_}, we have
\begin{eqnarray}\begin{split}\nonumber
\textrm{LHS(4.17)}&=4(d\sigma^{\star}+2b)E_{b,c}\left(u(t)\right)+(8+\varepsilon-2d\sigma^{\star}-4b)\left\| u(t)\right\| _{\dot{H}^{1}_{c}}^{2}\\
&\le4(1-\delta)(d\sigma^{\star}+2b)E_{b,\bar{c}}(W_{b,\bar{c}})+(8+\varepsilon-2d\sigma^{\star}-4b)(1+\delta')^{2}\left\| W_{b,\bar{c}}\right\| _{\dot{H}^{1}_{\bar{c}}}^{2}\\
&=\left\|W_{b,\bar{c}}\right\|_{\dot{H}^{1}_{\bar{c}}}^{2}\left[\frac{8(2-b)}{d-2}\left(1-\delta-(1+\delta')^{2}\right)
+\varepsilon(1+\delta')^{2}\right].
\end{split}\end{eqnarray}
Hence, by taking $\varepsilon>0$ small enough, we can get \eqref{GrindEQ__4_17_}.

First, we consider the case $xu_{0}\in L^{2}$ satisfying $E_{b,c}(u_{0})<E_{b,\bar{c}}(W_{b,\bar{c}})$ and $\left\|u_{0}\right\|_{\dot{H}^{1}_{c}}>\left\|W_{b,\bar{c}}\right\|_{\dot{H}^{1}_{\bar{c}}}$. Using the standard virial identity \eqref{GrindEQ__4_7_} and \eqref{GrindEQ__4_17_}, we have
\begin{equation}\nonumber
\frac{d^2}{dt^2}\left\|xu(t)\right\|_{L^{2}}^{2}=8\left\| u(t)\right\| _{\dot{H}^{1}_{c}}^{2} -\frac{4(d\sigma^{\star}+2b)}{\sigma^{\star}+2} \left\| |x|^{-b} |u|^{\sigma^{\star}+2} \right\|_{L^{1}}\le -c<0,
\end{equation}
which implies that the solution blows up in finite time.

Next, we consider the case $u_{0}$ is radial, and satisfies $E_{b,c}(u_{0})<E_{b,\bar{c}}(W_{b,\bar{c}})$ and $\left\|u_{0}\right\|_{\dot{H}^{1}_{c}}>\left\|W_{b,\bar{c}}\right\|_{\dot{H}^{1}_{\bar{c}}}$. Using the localized virial estimates \eqref{GrindEQ__4_11_}, we have
\begin{eqnarray}\begin{split}\nonumber
\frac{d^2}{dt^2}V_{\varphi_{R}}(t)\le& 8\left\| u(t)\right\| _{\dot{H}^{1}_{c}}^{2} -\frac{4(d\sigma^{\star}+2b)}{\sigma^{\star}+2} \left\| |x|^{-b} |u|^{\sigma^{\star}+2} \right\|_{L^{1}}\\
&+O\left(R^{-2}+\varepsilon^{-\frac{\sigma^{\star}}{4-\sigma^{\star}}}R^{-{\frac{2\left[(d-1)\sigma^{\star}+2b\right]}
{4-\sigma^{\star}}}}+\varepsilon\left\|u(t)\right\| _{\dot{H}^{1}_{c}}^{2}\right),
\end{split}\end{eqnarray}
for any $\varepsilon>0$ and any $t$ in the existence time. Taking $\varepsilon>0$ small enough and $R>1$ large enough depending on $\varepsilon$, it follows from \eqref{GrindEQ__4_17_} that
$$
\frac{d^2}{dt^2}V_{\varphi_{R}}(t)\le-c/2<0,
$$
which implies that the solution must blow up in finite time. This completes the proof.


\end{document}